\documentclass[intlimits]{tran-l}

\usepackage{enumerate}
\def\eq{equation}
\def\ea{eqnarray}

\def\tk{\widetilde{K(r)}}

\def\ep{\varepsilon}
\def\vphi{\varphi}
\def\Si{\Sigma}
\def\s{\sigma}
\def\o{\omega}

\def\S*{\Sigma_*}
\def\o{\omega}
\def\r{\rho}

\def\<{\langle}
\def\>{\rangle}
\def\j{R|J|}

\def\E{\mathbb{E}}
\def\N{\mathbb{N}}
\def\P{\mathbb{P}}
\def\R{\mathbb{R}}
\def\rd{\mathbb{R}^d}

\def\T{\mathcal{T}}

\def\H{\mathcal{H}}
\def\S{\mathcal{S}}
\def\PR{\mathcal{PR}}

\def\1{\mathbf{1}}

\newcommand{\diam}{\operatorname{diam}}

\newcommand{\esup}{\operatorname{ess\,sup}}
\newcommand{\einf}{\operatorname{ess\,inf}}
\newcommand{\elim}{\operatorname{ess\,lim}}
\newcommand{\id}{\operatorname{id}}
\newcommand{\Sim}{\operatorname{Sim}}

\newcommand{\Lip}{\operatorname{Lip}}
\newcommand{\nor}{\operatorname{nor}}
\newcommand{\Int}{\operatorname{int}}
\newcommand{\rea}{\operatorname{reach}}

\def\tit{\textit}
\theoremstyle{plain}
   \newtheorem{thm}{Theorem}[section]
   \newtheorem{thms}{Theorem}[subsection]
   \newtheorem{cors}[thms]{Corollary}
   
   \newtheorem{lems}[thms]{Lemma}

   \newtheorem{conds}[thms]{Condition}
\theoremstyle{remark}
\theoremstyle{definition}
   
   \newtheorem{defs}[thms]{Definition}

   \newtheorem{rems}[thms]{Remark}
   \newtheorem*{remarks}{Remarks}

\begin{document}
\title{Lipschitz-Killing curvatures of self-similar random fractals}
\author{M. Z\"ahle}
\address{University of Jena \\ Mathematical Institute}
\email{martina.zaehle@uni-jena.de}


\date{}

\dedicatory{Dedicated to Herbert Federer}


\begin{abstract}
For a large class of self-similar random sets $F$ in $\rd$ geometric parameters $C_k(F)$,  $k=0,\ldots,d$, are introduced. They arise as a.s. (average or essential) limits of the volume $C_d(F(\ep))$, the surface area $C_{d-1}(F(\ep))$ and the integrals of general mean curvatures over the unit normal bundles $C_k(F(\ep))$  of the parallel sets $F(\ep)$ of distance $\ep$ rescaled by $\ep^{D-k}$ as $\ep\rightarrow 0$. Here $D$ equals the a.s. Hausdorff dimension of $F$. The corresponding results for the expectations are also proved.
\end{abstract}

\maketitle
\section{Introduction}
\label{intro}
Self-similar sets in Euclidean space $\rd$ in the sense of Hutchinson \cite{Hu81} have intensively been studied in fractal analysis and geometry. Their probabilistic counterparts, the so-called (stochastically) self-similar random sets, were introduced independently by
Falconer \cite{Fa86}, Mauldin and Williams \cite{MW86} (1986), and Graf \cite{Gr87} (1987). Concerning the geometry of these random sets the literature up to now deals mainly with exact Hausdorff dimensions, multifractal spectra and associated measures. \\
In the present paper we will introduce a system of $d+1$ (random) geometric parameters which allow to distinguish between self-similar random sets $F$ with equal Hausdorff dimension $D$, but different geometric and topological features. This continues the work of Winter \cite{Wi08} who first investigated such parameters for deterministic self-similar sets with the open set condition and polyconvex neighborhoods.\\
Here we use the construction model from \cite{MW86} for such random $F$ satisfying the open set condition for a fixed deterministic open set with closure $J$. Our parameters are in close relationship to Federer's curvature measures for sets of positive reach \cite{Fe59}. For this we assume additionally that with probability 1 for Lebesgue almost all $r>0$ the parallel set $F(r)$ of amount $r$ has a Lipschitz boundary and the closure of its complement $\widetilde{F(r)}$ has positive reach. (We conjecture that this is already guaranteed by the strong open set condition on $\Int(J)$. At least this can be checked for many examples which do not possess polyconvex parallel sets. For $d\le 3$ it follows from a general result of Fu \cite{Fu85}). In this case the random sets $F(r)$ are Lipschitz $d$-manifolds of bounded curvature in the sense of \cite{RZ05} and their Lipschitz-Killing curvatures are determined by those of Federer:
$$C_k(F(r))=(-1)^{d-1-k}\, C_k\big(\widetilde{F(r)}\big)~,~~k=0,\ldots,d-2\, .$$
They are completed by the volume $C_d(F(r))$ and the surface area $C_{d-1}(F(r))$ without the additional assumption. The equalities for the total curvatures are localized to (signed) curvature measures. Our main result (see Theorem \ref{maintheorem} and Corollary \ref{averageconv}) is the following: Under an additional assumption on the model (only if $k\le d-2$) the limit
$$\lim_{\delta\rightarrow 0}\frac{1}{|\ln\delta|}\int_\delta^1\ep^{D-k}\E\big(C_k(F(\ep))\big)~\ep^{-1}d\ep $$
and the almost sure limit
$$\lim_{\delta\rightarrow 0}\frac{1}{|\ln\delta|}\int_\delta^1\ep^{D-k} C_k(F(\ep))~ \ep^{-1} d\ep$$
exist. Moreover, the first limit has an explicit integral representation and the second one is a random multiple of the first one. The multiplier is the inverse of a martingale limit related to the contraction ratios of the generating system of random similarities, i.e. it does not depend on $k$. If the logarithmic contraction ratios are non-lattice in a sense, then the average limits may be replaced by the ordinary limits as $\ep\rightarrow 0$. In general, the limits define random fractal Lipschitz-Killing curvatures of $F$, if they do not equal zero. For vanishing limits the correctness of the scaling exponents $D-k$ has to be checked in order to give such a curvature interpretation.\\
For the special case $k=d$ this concerns the (average) Minkowski content, and the results were proved by Gatzouras (see \cite{Ga00} and the references therein to related work).
We adopt his idea to apply the classical renewal theorem for the expectations and a renewal theorem of Nerman \cite{Ne81} for branching random walks in order to derive the almost sure limits.\\
At the same time we get an extension of Winter's results for the total curvatures in the deterministic case. However, Winter also proved such limit relationships for the corresponding curvature measures. The fractal versions in his case are all constant multiples of $D$-dimensional Hausdorff measure on $F$. It will be shown later that this remains valid under our conditions. In order to introduce random fractal curvature measures additional work is needed, in particular, related to exact Hausdorff measures of such random sets in the sense of Graf, Mauldin and Williams \cite{GMW88}.\\
The paper is organized as follows.\\
In Section 1 we summarize some background from classical singular curvature theory. The Appendix provides an auxiliary new result result in this direction - an estimate for the Lipschitz-Killing curvatures of sufficiently large parallel set of an arbitrary compact subset of $\rd$, which is related to well-known isodiametric inequalities in the convex case.\\
Section 2.1 provides the random iterated function scheme and the limit set $F$ by means of random trees and an associated branching random walk.\\
In Section 2.2 we follow Gatzouras \cite{Ga00} and present a slight extension of the corresponding renewal theorem for branching random walks.\\
Its application to fractal curvatures is prepared in Section 2.3, which contains the formulation of the main results.\\
The key proofs are given in Section 3 using the above mentioned theorem from the Appendix.

\section{Background from classical singular curvature theory - curvature measures of parallel sets}\label{sec:classcurv}

The geometry of classical geometric sets in $\rd$ may be described by certain parameters forming complete systems of Euclidean invariants in the following sense: It is well known from convex geometry (Hadwiger's characterization theorem) that every set-additive, continuous and motion invariant functional on the space of convex bodies is a linear combination of Minkowski's quermass integrals. The latter are also called intrinsic volumes. For smooth boundaries they agree - except the volume - with the integrals of the elementary symmetric functions of principal curvatures, the so-called integrals of (higher order) mean curvatures of smooth manifolds. Federer \cite{Fe59} unified and extended the intrinsic volumes and the integrals of mean curvatures with tools of geometric measure theory: He introduced curvature measures of sets with positive reach by means of a Steiner polynomial for the volume of parallel set. (This is related to Weyl's tube formula in differential geometry, where the mean curvatures arise as traces of certain powers of the Riemannian curvature tensor.) Explicit representations of Federer's curvature measures in form of integrating Lipschitz-Killing curvature forms or symmetric functions of generalized principal curvatures over the associated unit normal bundle were given in \cite{Za86b}. Starting from the 1980s (and earlier in convex geometry) up to now additive extensions and other versions of these curvature measures for classes of singular sets in $\rd$ have been studied with methods of geometric measure theory and algebraic geometry. Nowadays the notion {\it Lipschitz-Killing curvature measures} is used.\\
Localizations of Hadwiger's characterization theorem to curvature measures of convex sets where given by Schneider \cite{Sch78}. In \cite{Za90} this was used  together with an appropriate notion of continuity as approximation tool in order to generalize these characterizations to large classes of singular sets. At the same time this served as a motivation to study such curvature properties also for fractals.\\
The first results in this direction were obtained by Winter \cite{Wi06}, \cite{Wi08}, who worked out essential tools for investigating self-similar fractal sets in $\rd$ under this point of view. Under the usual open set condition and the additional assumption of polyconvex neighborhoods he solved the corresponding problems completely. In view of the Gauss-Bonnet theorem the notion of fractal Euler number investigated before in Llorente and Winter \cite{LW07} under some additional assumption may be considered as a special case.\\
In order to avoid the condition of polyconvex parallel sets and to extend such curvatures to random fractals we suggest now another approach.
For certain classes of compact sets $K\subset\rd$ (including many classical singular sets) it turns out that for Lebesgue-almost all distances $r>0$ the parallel set
\begin{\eq}\label{parallelset}
K(r):=\{x\in\rd: d(x,K)=\min_{y\in K}|x-y|\le r\}
\end{\eq}
possesses the property that the closure of its complement
\begin{\eq}\label{closcompl}
\widetilde{K(r)}:= \overline{K(r)^c}
\end{\eq}
is a set of positive reach with Lipschitz boundary. A sufficient condition is that $r$ is a regular value of the Euclidean distance function to $K$ (see Fu \cite[Theorem 4.1]{Fu85} together with \cite[Proposition 3]{RZ03}). In this case
both the sets $\widetilde{K(r)}$ and $K(r)$ are {\it Lipschitz d-manifolds of bounded curvature} in the sense of \cite{RZ05}, i.e., their {\it k-th Lipschitz-Killing curvature measures}, $k=0,1,\ldots,d-1$, are determined in this general context and agree with the classical versions in the special cases.
Moreover, they satisfy
\begin{\eq}
C_k(K(r),\cdot)=(-1)^{d-1-k}C_k\big(\tk,\cdot\big)\, .
\end{\eq}
Therefore the $C_k(K(r),\cdot)$ are signed measures with finite {\it variation measures}\\
$C_k^{var}(K(r),\cdot)$ and the explicit integral representations are reduced to \cite{Za86b} (cf. \cite[Theorem 3]{RZ05} for the general case). In the present paper the normal cycle representation is needed only in the Appendix. We will briefly mention the normal bundle and the current construction at the end of the section. (The reader not familiar with the corresponding geometric integration theory may skip the current theoretical part.) Here we will list the main properties of the curvature measures for the parallel sets as above which will be used repeatedly:\\
$C_{d-1}(K(r),\cdot)$ agrees with $(d-1)$-{\it dimensional Hausdorff measure} $\H^{d-1}$ on the boundary $\partial K(r)$. The latter is a bounded measure for all $r>0$ and all compact sets $K$ to which it is applied below. Therefore we use this notation in any case. Furthermore, for completeness we define
$C_d(K(r),\cdot)$ as {\it Lebesgue measure restricted to} $K(r)$. The {\it total measures (curvatures)} of $K(r)$ are denoted by
\begin{\eq}
C_k(K(r)):=C_k(K(r),\rd)\, ,~~ k=0,\ldots ,d\, .
\end{\eq}
By an associated Gauss-Bonnet theorem (see \cite[Theorems 2,3]{RZ03}) the {\it Gauss curvature}
$C_0(K(r))$ coincides with the {\it Euler-Poincar\'{e} characteristic} $\chi(K(r))$.\\
The curvature measures are {\it motion invariant}, i.e.,
\begin{\eq}
C_k(g(K(r)),g(\cdot))=C_k(K(r),\cdot)~~\mbox{for any Euclidean motion}~g\, ,
\end{\eq}
they are {\it homogeneous of degree} $k$, i.e.,
\begin{\eq}
C_k(\lambda K(r),\lambda (\cdot))=\lambda^k\, C_k(K(r),\cdot)\, ,~~\lambda>0\, ,
\end{\eq}
and locally determined, i.e.,
\begin{\eq}
C_k(K(r),(\cdot)\cap G)=C_k(K'(r'),(\cdot)\cap G)
\end{\eq}
for any open set $G\subset\rd$ such that $K(r)\cap G=K'(r')\cap G$, where $K(r)$ and $K'(r')$ are
both parallel sets where the closures of the complements have positive reach.\\

We now summarize some facts about sets with positive reach needed below:\\
Recall that $\rea X$ of a set $X\subset \rd$ is defined as the supremum over all $s>0$ such
that every point $x$ in the $s$-parallel set of $X$ there is a unique point $\Pi_Xx\in X$ nearest to $x$. The mapping $\Pi_X$ (on its domain) is called the metric projection onto $X$. For a set X of positive reach the {\it unit normal bundle} is defined as
$$\nor X:=\{(x,n)\in \rd\times S^{d-1}: x\in X ,\, n\in \operatorname{Nor}(X,x)\}$$
where $\operatorname{Nor}(X,x)$ is the dual cone to the (convex) tangent cone of $X$ at $x$.\\
If additionally $\nor X\cap\r(\nor X)=\emptyset$ for the {\it normal reflection} $\r\: (x,n)\mapsto (x,-n)$, then $X$ is full dimensional with Lipschitz boundary (see \cite[Proposition 3]{RZ03}).\\
For general $X$ with $\rea X>0$ there is an associated  rectifiable current called the unit normal cycle of $X$ which is given by
$$N_X(\vphi):=\int_{\nor X}\<a_X(x,n),\vphi(x,n)\>\H^{d-1}(d(x,n))$$
for an appropriate unit simple $(d-1)$-vector field $a_X=a_1\wedge\ldots\wedge a_{d-1}$ associated a.e. with the tangent spaces of $\nor X$ and for integrable differential $(d-1)$-forms $\vphi$. In these terms for $k\le d-1$ the curvature measure may be represented by
$$C_k(X,B)=N_X\llcorner \1_{B\times\rd}(\vphi_k)=\int_{\nor X\cap (B\times\rd)}\<a_X(x,n),\vphi_k(n)\>\H^{d-1}(d(x,n))$$
for any bounded Borel set $B\subset\rd$, where the $k$-th Lipschitz-Killing curvature form $\vphi_k$ does not depend on the points $x$ and is defined by its action on a simple $(d-1)$-vector $\eta=\eta_1\wedge\ldots\wedge\eta_{d-1}$ as follows: Let $\pi_0(y,z):=y$ and $\pi_1(y,z):=z$ be the coordinate projections in $\rd\times\rd$, $e'_1,\ldots,e'_d$ be the dual basis of the standard basis in $\rd$ and $\mathcal{O}_k$ the surface area of the $k$-dimensional unit sphere. Then we have
$$\<\eta,\vphi_k(n)\>:=\mathcal{O}_{d-k}^{-1}\sum_{\ep_i\in\{0,1\},\, \sum\ep_i=d-1-k}\<\pi_{\ep_1}\eta_1\wedge\ldots\wedge\pi_{\ep_{d-1}}\eta_{d-1}\wedge n,e'_1\wedge\ldots\wedge e'_d\>\, .
$$

\section{Self-similar random fractals - the model and statement of the main results}
\label{sec:ssrfractals}
\subsection{Random recursive constructions and associated branching random walks}\label{ssec:recconstr-branwalks}
We briefly describe the random recursive construction model introduced in Mauldin and Williams \cite{MW86} (and independently in Falconer \cite{Fa86} and Graf \cite{Gr87} with different methods). Additionally we use some ideas from Gatzouras \cite{Ga00} for relationships of associated random functions to branching random walks.
Let $\Sim$ be the set of contracting similarities and $J$ a nonempty compact subset of of $\rd$ with $J=\overline{\Int(J)}$. Our basic object is a random element
  $\S:=\{S_1,\ldots,S_\nu\}$ with $S_i\in\Sim$ if $\nu>0$, and
 $\S:=\{\id\}$ if $\nu=0$, where $\nu$ is a random variable with values in $\N_0:=\{0\}\cup\N$. We suppose that $\S$ satisfies the \tit{open set condition} (briefly (OSC)) with respect to $\Int(J)$:
%
\begin{equation}\label{OSC}
\bigcup_{i=1}^\nu S_i(\Int(J))\subset \Int(J)\quad\mbox{and}\quad
S_i(\Int(J))\cap S_j(\Int(J))=\emptyset~, i\ne j\, ,
\end{equation}
with probability 1. Denote the corresponding probability space by $[\Omega^0,\mathfrak{F}^0,P^0]$. $\S$ is also called \tit{random iterated function system}, briefly RIFS.
The corresponding random fractal set is introduced by means of the \tit{code space} $\Si :=\N^\N$ and a \tit{random Galton-Watson tree} in the set of all finite sequences $\Si_*:=\{ 0\}\cup\bigcup_{n=1}^\infty \N^n$:\\
For $\s=\s_1\ldots\s_k\, ,\tau=\tau_1\dots\tau_l\in\Si_*$ we write $|\s|:=k$ for the \tit{length} of $\s$, $\s|i:=\s_1\ldots\s_i$, i<k, for the \tit{restriction} to the first $i$ components, and $\s\tau:=\s_1\ldots\s_k\tau_1\dots\tau_l$ for the \tit{concatenation} of $\s$ and $\tau$. (We will use analogous notations for infinite $\s\in\Si$, resp. $\tau\in\Si$.) By convention, $0\s=\s$.\\
For each $\s\in\Si_*$ let $[\Omega_\s,\mathfrak{F}_\s,\P_\s]$ be a copy of the above probability space. The basic probability space for the random construction model is the product space
\begin{equation}\label{basicprobspace}
[\Omega,\mathfrak{F},\P]:=\bigotimes_{\s\in \Si_*}[\Omega_\s,\mathfrak{F}_\s,\P_\s]\, .
\end{equation}
On this space a family of independent identically distributed RIFS\\
$\{\S_\s=\{S_{\s 1},\ldots,S_{\s\nu_\s}\}\}_{\s\in\Si_*}$ (where $\S_\s=\{\id\}$ if $\nu_\s=0$)
 with i.i.d. random numbers $\{\nu_\s\}_{\s\in\Si_*}$ is then determined. Let $\mathfrak{F}_n$ be the $\s$-algebra generated by all $\S_\s$ and $\nu_\s$ with $|\s|\le n$. For brevity we write
\begin{eqnarray*}
\bar{S}_\s & := & S_{\s|1}\circ S_{\s|2}\circ\cdots\circ S_{\s||\s|}\, ,\\
r_\s & := & \Lip (S_\s)\, ,\\
\bar{r}_\s & := & \Lip(\bar{S}_\s)=r_{\s|1}r_{\s|2}\cdots r_{\s||\s|}\, ,\\
J_\s & := & \bar{S}_\s J\, ,\\
K_\s & := & K\cap J_\s\, ,\\
K^\s & := & \bar{S}_\s^{-1}(K)\cap J\, ,\\
|K| & := & \diam(K)\, ,
\end{eqnarray*}
$\s\in\Si_*$, for any compact set $K$. Note that $K_\s$ is a random compact subset of $J$ in the sense of stochastic geometry (measurable with respect to the Borel $\s$-algebra given by the Hausdorff distance).\\
Set $\T_0:=\{0\}$ and define inductively $\T_{n+1}:=\emptyset$, if $\T_n=\emptyset$, and
$$\T_{n+1}:= \{\s i:\, \s\in\T_n,\, \nu_\s\ne0,\, 1\le i\le\nu_\s\}\, ,$$
if $\T_n\ne\emptyset$. Then
$$\T:=\bigcup_{n=0}^\infty \T_n$$
is the \tit{population tree of a random Galton-Watson process}. $\T_n$ represents the family of individuals in the $n$-th generation with ancestor $0$. The \tit{boundary} of $\T$ is defined by
$$\partial \T:=\{\s\in\Si : \s|n\in\T, n\in\N\}\, .$$
In the sequel we consider the \tit{supercritical case} with
\begin{equation}
1<\mathbb{E}\nu<\infty\, ,
\end{equation}
where this boundary is nonempty with positive probability (see the classical literature on branching processes). The random compact set
%
\begin{equation}
\label{SSRS}
F:=\bigcap_{n=1}^\infty\bigcup_{\s\in\T_n} J_\s
\end{equation}
is the associated \tit{self-similar random set}. $F$ is the image of the boundary $\partial\T$ under the random \tit{projection}
$$\pi: \s\mapsto\lim_{n\rightarrow\infty}\bar{S}_{\s|n}(x_0)$$
for an arbitrary starting point $x_0\in\rd$. By construction
the random set $F$ is non-empty with positive probability. Its {\it stochastic self-similarity property} reads here as follows:
%
\begin{equation}\label{statselfsim}
F=\bigcup_{S_i\in \S}S_i(F^i)
\end{equation}
where the random sets $F^i,\, i\in\N_0$, are independent, have the same distribution as $F$, and the random element $\S=\{S_1,\ldots,S_\nu\}$ with contracting similarities $S_i$, if $\nu>0$, and
$S=\{\id\}$, if $\nu=0$, is as above and independent of the $F^i$. \\
It is well-known that with probability 1 (briefly w.p.1) the \tit{Hausdorff dimension} $D$ of the self-similar random set $F$ is uniquely determined by the equation
%
\begin{equation}\label{Hausdorffdim}
\E\left(\sum_{i\in\T_1} r_i^D\right) = 1\, .
\end{equation}
Under the additional assumption that $\P(F\cap \Int(J)\ne\emptyset)>0$, the so-called \tit{strong open set condition} for the open set $\Int(J)$ from OSC \eqref{OSC}, this has been proved in \cite{MW86}, \cite{Fa86} and \cite{Gr87}. Following Patzschke \cite{Pa97} it remains true supposing only OSC.\\
We are interested in curvature properties of the random fractal set $F$ . The main tool will be approximation by parallel neighborhoods of small distances using the Lipschitz-Killing curvature measures from Section 1 and suitable rescalings. (For the case of deterministic similarities $(S_1,\ldots,S_N)$ and the assumption of polyconvex neighborhoods for the deterministic self-similar set $F$ the corresponding notions and results have been worked out in the Thesis of Winter, see \cite{Wi08}.)
The related problem for the Minkowski content solved in \cite{Ga00} may be considered as a marginal case. An important tool will be again the \tit{branching random walk} $\{W_\s\}_{\s\in\Si_*}$ defined by the recursive formula
%
\begin{equation}\label{branwalk}
W_\s=W_{\s|(|\s|-1)}+\ln r_\s^{-1}\, ,
\end{equation}
if $\s\in\T$, and $W_\s:=\infty$, if $\s\in\Si_*\setminus\T$.
(Recall that $r_\s$ denotes the \tit{contraction ratio} of the similarity $S_\s$ for $\s\in\T$, and by convention $r_0=1$.) In particular, $W_0=0$. As in \cite{Ga00} the classical renewal theorem essentially used for the above problems in the deterministic case will be replaced by an associated stochastic version.
\subsection{Renewal theorem for branching random walks}\label{ssec:renew-branwalk}
We refer to Gatzouras \cite[Section 3.2]{Ga00}.
First consider the associated nonnegative martingale
%
\begin{equation}\label{martingale}
M_n:=\sum_{\s\in\T_n} e^{-DW_\s}= \sum_{\s\in\T_n}\bar{r}_\s^D\, ,\quad  n\ge 0\, ,
\end{equation}
with respect to the filtration $\{\mathfrak{F}_n\}_{n\ge 0}$. (Note that by \eqref{Hausdorffdim} $\E M_1=1$.) According to the martingale convergence theorem the limit
%
\begin{equation}\label{eq:mart-conv}
M_\infty:=\lim_{n\rightarrow\infty} M_n
\end{equation}
exists w.p.1. The next theorem is due to Biggins \cite{Bi77}; see Lyons \cite{Ly97} for a conceptual proof using Lyons, Pemantle and Peres \cite{LPP95}.\\

\begin{thms}[Biggins]\label{Biggins}
The following are equivalent:
\begin{enumerate}
\item [{\rm (i)}] $\E(M_1\ln^+M_1)<\infty$
\item [{\rm (ii)}]$\P(M_\infty=0)<1$
\item [{\rm (iii)}]$\P(M_\infty>0\, |\,\mbox{non-extinction})=1$
\item [{\rm (vi)}]$\E M_\infty=1$
\item [{\rm (v)}]$M_n\rightarrow M_\infty~\mbox{in}~L^1$\, .
\end{enumerate}
\end{thms}
Recall that the underlying probability space $[\Omega,\mathfrak{F},\P]=\bigotimes_{\s\in \Si_*}[\Omega_\s,\mathfrak{F}_\s,\P_\s]$ is a product space. For each $\tau\in\Si_*$ define the \tit{shift operator}
$\theta_\tau:\Omega\rightarrow\Omega$ by
\begin{equation}\label{shift}
(\theta_\tau\o)_\s:=\o_{\tau\s}\, .
\end{equation}
Besides the branching random walk $\{W_\s\}_{\s\in\T}$ from \eqref{branwalk} we now consider a basic stochastic process $Y$ satisfying the following.
\begin{conds}\label{a.e.continuousw.p.1}
$Y=\{Y_t: t\in\R\}$ is a real-valued measurable stochastic process  on $[\Omega,\mathfrak{F},\P]$ vanishing for $t<0$, which is continuous a.e. with probability $1$, i.e., there exists a set $C\in\mathfrak{B}(\R)\bigotimes\mathfrak{F}$ such that
$\int_\Omega\int_\R\1_{C^c}(t,\o)\, d t\, \P(d\o)=0$ and for any $(t,\o)\in C$ the function $Y_{(\cdot)}(\o)$ is continuous at point $t$.
\end{conds}
The process $Y$ induces a family of i.i.d. copies defined by
$$Y^\s_t(\o):=Y_t(\theta_\s\o)\, ,\quad \s\in\Si_*\, .$$
Then we can introduce the \tit{branching process} associated with $W$ and $Y$ in the sense of Jagers \cite{Ja75}:
\begin{equation}\label{branchingprocess}
Z_t:=\sum_{\s\in\T}Y^\s_{t-W_\s}\, .
\end{equation}
We are interested in the limit behavior of the process $e^{-Dt}Z_t$ as $t\rightarrow\infty$. Here the measure
\begin{equation}\label{latt-nonlatt-meas}
\mu:=\E\left(\sum_{i\in\T_1}\1_{(\cdot)}(W_i)\right)=\E\left(\sum_{i\in\T_1}\1_{(\cdot)}(|\ln r_i|)\right)
\end{equation}
on $\R$ plays a crucial role. Note that $\mu(\R)=\E\nu$ and recall the assumption $1<\E\nu<\infty$.
Denote
\begin{equation}\label{limitconstant}
\lambda(D):=\E\left(\sum_{i\in\T_1}W_i e^{-DW_i}\right)=\E\left(\sum_{i\in\T_1}|\ln r_i|\, r_i^D\right)\, .
\end{equation}
The following \tit{renewal theorem for the branching process} $Z$ is essential for our purposes. (Recall the random martingale limit $M_\infty$ from \
\eqref{eq:mart-conv}.) Note that the essential limit as $t\rightarrow 0$ is meant w.r.t. Lebesgue measure and is defined as the common value of $\lim\esup$ and $\lim\einf$, if the latter coincide.\\

\begin{thms}\label{renewal-branwalk}
Suppose that the process $Y$ satisfies Condition \ref{a.e.continuousw.p.1} and there exists a non-increasing integrable function $h: [0,\infty)\rightarrow(0,\infty)$, such that
\begin{equation}\label{boundedness-condition}
\E\left(\esup_{t\ge 0}\limits\frac{e^{-Dt}|Y_t|}{h(t)}\right)<\infty\, .
\end{equation}
\begin{enumerate}
\item[{\rm (i)}] {\rm [Nerman]} If the measure $\mu$ is non-lattice then
$$\elim_{t\to\infty}\limits e^{-Dt}\E(Z_t)=\frac{1}{\lambda(D)}\int_0^\infty e^{-Ds}\E(Y_s)\, ds$$
and
$$\elim_{t\to\infty}\limits e^{-Dt}Z_t=\frac{M_\infty}{\lambda(D)}\int_0^\infty e^{-Ds}\E(Y_s)\, ds \quad w.p.1\, .$$
\item[{\rm (ii)}] {\rm [Gatzouras]} If the measure $\mu$ is lattice with lattice constant $c$, then for Lebesgue-a.e. $s\in[0,c)$ we have
   $$\lim_{n\rightarrow\infty}e^{-D(s+nc)}\E(Z_{s+nc})=\frac{1}{\lambda(D)}\sum_{n=0}^\infty e^{-D(s+nc)}\E(Y_{s+nc})\, .$$
   and
   $$\lim_{n\rightarrow\infty}e^{-D(s+nc)}Z_{s+nc}=\frac{M_\infty}{\lambda(D)}\sum_{n=0}^\infty e^{-D(s+nc)}\E(Y_{s+nc})\quad w.p.1\, .$$
\end{enumerate}
\end{thms}
Assertion (i) is shown in Nerman \cite{Ne81} for the case of Skorohod-regular processes (which are a.e. continuous w.p.1). The lattice case (ii) is derived in \cite{Ga00} from Nerman's proof of (i). Note that these proofs remain valid under our Condition \ref{a.e.continuousw.p.1} and (\ref{boundedness-condition}) when considering essential limits. The convergence of the expectations may be considered as a special case of the classical renewal theorem for deterministic functions. (Feller's proof in \cite{Fel66} works also for the essential limits.) A straightforward consequence of (ii) is the following.\\

\begin{cors}\label{averageconv}
Suppose that the conditions of Theorem \ref{renewal-branwalk} {\rm (ii)} are satisfied and
\begin{equation}\label{supremum}
\int_0^c \sup_n \left(e^{-D(s+nc)}\E\left|Z_{s+nc}\right|\right)\, ds <\infty\, .
\end{equation}
Then we have
$$\lim_{T\rightarrow\infty}\frac{1}{T}\int_0^T e^{-Dt}\E(Z_t)\, dt=\frac{1}{\lambda(D)}\int_0^\infty e^{-Ds}\E(Y_s)\, ds$$
and
$$\lim_{T\rightarrow\infty}\frac{1}{T}\int_0^T e^{-Dt}Z_t\, dt=\frac{M_\infty}{\lambda(D)}\int_0^\infty e^{-Ds}\E(Y_s)\, ds\quad  w.p.1\, .$$
\end{cors}
\begin{proof}
We show the almost sure convergence. The arguments for the expectations are similar.\\ The branching process $Z$ is measurable with respect to $\mathfrak{B}(\R)\otimes\mathfrak{F}$ and therefore
$$E:=\left\{ (s,\o)\in [0,\infty)\times\Omega: \lim_{n\rightarrow\infty}e^{-D(s+nc)}Z_{s+nc}(\o)=L_s\right\}
\in\mathfrak{B}(\R)\otimes\mathfrak{F}\, ,$$
where $L_s$ denotes the right hand side of (ii). Theorem \ref{renewal-branwalk} {\rm (ii)}
and Fubini yield
$$0=\int_0^c \int_\Omega \1_{E^c}(s,\o)\, \P(d\o)\, ds = \int_\Omega \int_0^c \1_{E^c}(s,\o)\, ds\, \P(d\o)\, ,$$
i.e., w.p.1 we get $\lim_{n\rightarrow\infty}e^{-D(s+nc)}Z_{s+nc}=L_s$ for a.e. $s\in [0,c)$.
Taking into regard \eqref{supremum} (where we could omit the expectation for convergence w.p.1) and dominated convergence we infer
$$\lim_{n\rightarrow\infty}\int_0^c e^{-D(s+nc)}Z_{s+nc}\, ds=\int_0^c L_s\, ds$$
w.p.1. Instead of the sequence on the left hand side we may also take its arithmetic means which converge to the same limit. The latter is equal to the right hand side of the assertion as an easy calculation shows. Similarly, the limit of the arithmetic means is the same as the limit of the Ces\'{a}ro means on the left hand side of the assertion.
\end{proof}
\subsection{Application to fractal curvatures}\label{ssec:applfraccurv}
We now turn back to the self-similar random set $F$ from \eqref{SSRS}. Let $\mathcal{K}$ be the space of non-empty compact subsets of
our primary compact set $J$. $\mathfrak{B}(\mathcal{K})$ denotes the Borel $\s$-algebra with respect to the Hausdorff distance $d_H$ on $\mathcal{K}$.
We further consider the space $\mathcal{F}$ of closed subsets of $\rd$ provided with the hit-and-miss topology (generated by the sets $\{A\in\mathcal{F}:A\cap O\ne\emptyset\}$ and $\{A\in\mathcal{F}:A\cap C=\emptyset\}$ for open $O$ and closed $C$) and the Borel-$\s$-algebra $\mathfrak{B}(\mathcal{F})$. The topology restricted to $\mathcal{K}$ is generated by the metric $d_H$.\\
(Recall the notations \eqref{parallelset} and \eqref{closcompl}
for parallel sets and the closures of their complements. It is easy to see, that the mapping
$$(r,K)\mapsto \widetilde{K(r)} \quad\mbox{from}\quad [0,\infty)\times \mathcal{K} \quad\mbox{to}\quad \mathcal{F}$$
is $(\mathfrak{B}([0,\infty))\otimes \mathfrak{B}(\mathcal{K})\, ,\, \mathfrak{B}(\mathcal{F}))$-measurable.\\
Let $\PR$ be the space of subsets of $\rd$ with positive reach (cf. the end of Section \ref{sec:classcurv}). According to \cite[Proposition 1.1.1]{Za86a} we get $\PR\in\mathfrak{B}(\mathcal{F})$. Moreover, the mapping $X\mapsto \nor X$ from $\PR$ into the space of closed subsets of $\rd\times\rd$ is measurable. (Recall the definition of the unit normal bundle $\nor X$ from Section \ref{sec:classcurv}.)
 This implies the following measurability property (for the {\it normal reflection} $\rho:(x,n)\mapsto (x,-n)$):\\
\begin{lems}\label{lem:measurable-posreach}
\begin{eqnarray*}
Reg:&=&\bigg\{(r,K)\in [0,\infty)\times\mathcal{K}: \widetilde{K(r)}\in\PR ,\, \nor \widetilde{K(r)}\cap\rho(\nor\widetilde{ K(r)}) =\emptyset \bigg\}\\
& &\in\mathfrak{B}([0,\infty))\otimes \mathfrak{B}(\mathcal{K})\, .
\end{eqnarray*}
\end{lems}
The elements of $Reg$ will be called {\it regular pairs}.\\
In Fu \cite[Theorem 4.1]{Fu85} it is shown that in space dimensions $d\le 3$ for any compact set $K$ there exists a bounded exceptional set $N$ of Lebesgue measure $0$ such that for any $r\notin N$ the set $\widetilde{K(r)}$ has positive reach and $\nor \widetilde{K(r)}\cap\rho(\nor\widetilde{ K(r)})=\emptyset$. Moreover, if $r>\sqrt{d/(2d+2)}|K|$ both these assertions hold for {\it any} space dimension $d$. This basic result, the preceding lemma and Fubini imply the following property of the self-similar random set $F$:\\
\begin{cors}
In space dimensions $d\le 3$ w.p.1 for Lebesgue-a.e. $r>0$ the pair $(r,F)$ is regular in the above sense.
\end{cors}

In higher dimensions we shall formulate this as a \tit{main geometric condition} on the self-similar random set $F$:\\
\begin{defs}\label{def:regular SSRS}
The self-similar random set $F$ is called {\it regular} if w.p.1 for Lebesgue-a.e. $r>0$ the pair $(r,F)$ is regular.
\end{defs}
For regular pairs $(r,F)$ the Lipschitz-Killing curvature measures of $F(r)$ are determined by means of those for $\widetilde{F(r)}$ (cf. Section \ref{sec:classcurv}):
%
\begin{equation}\label{eq:randcurvmeas-parall}
C_k(F(r), \cdot)=(-1)^{d-1-k}C_k(\widetilde{F(r)},\cdot)\, ,\quad k=0,\ldots,d-1\, .
\end{equation}
According to Fu's result mentioned above this holds, in particular, for any realization of $F$ and $r>r_d$ where
\begin{equation*}\label{def:r_d}
r_d:=\sqrt{d/(2d+2)}\, |J|\, .
\end{equation*}
For the exceptional pairs we set $$C_k(F(r),\cdot):=0\, .$$
Then the $C_k(F(r),\cdot)$ are {\it random signed measures}. The corresponding measurability properties follow from \cite[Theorem 2.1.2]{Za86a} and Lemma \ref{lem:measurable-posreach}:
If the self-similar random set $F$ is regular, the mapping
$(r,F)\mapsto C(F(r),B)$  for any Borel set $B\in\rd$ is product measurable. Moreover, one obtains the following {\it continuity property}:\\
\begin{lems}\label{lem:contcurvmeas}
For any $(r_0,K)\in Reg$ and $r\rightarrow r_0>0$ we have $(r,K)\in Reg$ and the measures $C_k(K(r),\cdot)$ weakly converge to $C_k(K(r_0),\cdot)$.
\end{lems}
The proof follows from \cite[Propositions 1 and 3]{RZ03} and \cite[Proposition 6]{RZ05}.
(Recall that if $(r,K)\in Reg$ both the sets $K(r)$ and $\widetilde{K(r)}$ are Lipschitz $d$-manifolds of bounded curvature in the sense of \cite{RZ05}.)\\
This implies the following auxiliary result.\\
\begin{cors}\label{auxcontinuity}
If the self-similar random set $F$ is regular, then the random function $\Phi(r):=C_k(F(r))$ is a.e. continuous w.p.1.
\end{cors}
\vskip3mm

\begin{rems}
If $k=d$ this property remains valid for general compact sets and all $r>0$, i.e. we need not restrict to the class $\PR$. For $k=d-1$ the set $\PR$ can be replaced by the set of $(\mathcal{H}^{d-1},d-1)$-rectifiable closed subsets of $\rd$ in the sense of Federer \cite{Fe69}. The above measurability and continuity properties for this case are treated, e.g., in \cite{Za82}. (Here the weak convergence of $C_{d-1}(K(r),\cdot)$ to $C_{d-1}(K(r_0),\cdot)$ as $r\rightarrow r_0$ is a well-known result from geometric measure theory.)
\end{rems}

In the Appendix it will be shown that parallel sets of distances greater than $\sqrt{2}|J|$ are nice sets concerning  their regularity properties. Therefore in the sequel we fix an arbitrary $$R>\sqrt{2}\, .$$
It turns out that the relevant formulas below do not depend on the choice of $R$.\\
In order to formulate the remaining conditions on $F$ and to apply the renewal theorem we now turn back to the tree
construction from Section \ref{ssec:recconstr-branwalks}:\\
For fixed $k\in\{0,1,\ldots d\}$ we consider the {\it basic stochastic processes}
\begin{equation}\label{curvmeas-renewalprocess}
Z_t:=(\j)^{-k}e^{kt}\, C_k(F(\j e^{-t}))\1_{[0,\infty)}(t)
\end{equation}
and
\begin{equation}\label{curvmeas-auxprocess}
Y_t:=Z_t-\sum_{i\in\T_1}Z_{t-W_i}^i\, .
\end{equation}
Iterating the last equation we get for $Z_t$ the branching process representation \eqref{branchingprocess}: $$Z_t=\sum_{\s\in\T}Y_{t-W_\s}^\s\, .$$
(Recall the notation $X^\s_t(\o)=X_t(\theta_\s(\o))$, $\s\in\Si_*$, for a process $X$ and the shift operator \eqref{shift} on the probability space \eqref{basicprobspace}.) Here we have in mind the scaling property and the additivity of the Lipschitz-Killing curvature measure $C_k$ in order to obtain the conditions on the process $Y$ for the renewal theorem. The continuity condition \ref{a.e.continuousw.p.1} will follow from the regularity of the self-similar random set $F$ if $k\le d-2$ or from the rectifiability of the boundaries of its parallel sets if $k=d-1$. For the boundedness condition \eqref{boundedness-condition} we will use an additional assumption if $k\le d-2$, which is formulated in the language of random stoppings in the tree construction:\\
For any $r>0$ we define the {\it random subtree}
\begin{equation}\label{subtree-T(r)}
\T(r):=\{\s\in\T:  \j\, \bar{r}_\s \le r <  \j\, \bar{r}_{\s||\s| -1}\}\, .
\end{equation}
This is a so-called {\it Markov stopping} on our probability space. \\
Recall the above notations
 $F^\s=\bar{S}_\s^{-1}(F)\cap J=F\circ\theta_\s$
for the shift operator $\theta$ from \eqref{shift}. Then we obtain from the representation
 $F=\bigcup_{\s\in\T(r)}F_\s$ with $F_\s=F\cap\bar{S}_\s J$
  the stochastic self-similarity
\begin{equation}\label{self-sim-Markov-stop}
F=\bigcup_{\s\in\T(r)}\bar{S}_\s(F^\s)\, ,
\end{equation}
where $F^\s,\, \s\in\Si_*$, are i.i.d. copies of $F$ and the random similarities
$\bar{S}_\s,\, \s\in\Si_*$, and the random subtree $\T(r)$ are as before and independent of the $F^\s$.
At the same time we get for any $r>0$ the representation
$$F(r)=\bigcup_{\s\in\T(r)}\bar{S}_\s\bigg(F^\s\big(\frac{r}{\bar{r}_\s}\big)\bigg)\, .$$
For any $\s\in\T(r)$ we have $\frac{r}{\bar{r}_\s}\ge R|J|$ and $F^\s\subset J$. Then Theorem \ref{thm-appendix} in the Appendix implies that $\partial\big(F^\s\big(\frac{r}{\bar{r}_\s}\big)\big)$ is a Lipschitz $(d-1)$-submanifold. Since the number of $\s\in\T(r)$ is finite, we have shown the following:\\

\begin{lems}
With probability 1 for any $r>0$ the boundary of the random set $F(r)$ is $(\mathcal{H}^{d-1},d-1)$-rectifiable.
\end{lems}

(For the deterministic case see Rataj and Winter \cite{RW09}.)\\
Furthermore, we introduce the subset of those words $\s$ in $\T(r)$ for which the set $F_\s(r)=(F\cap J_\s)(r)$ has distance less than $r$ to the {\it boundary} of the first iterate $SJ:=\bigcup_{i\in\T_1}J_i$ of the basic set $J$ under the random similarities:
\begin{equation}\label{subtree-T_b(r)}
\T_b(r):=\left\{\s\in\T(r): F_\s(r)\cap(SJ)^c(r)\ne\emptyset\right\}\, .
\end{equation}

We now can formulate the {\it main theorem} of the paper.\\ Recall the martingale limit $M_\infty=\lim_{n\rightarrow\infty}\sum_{\s\in\T_n}\bar{r}_\s^D$ from \eqref{eq:mart-conv},
the measure\\ $\mu:=\E\left(\sum_{i\in\T_1}\1_{(\cdot)}(|\ln r_i|)\right)$ from \eqref{latt-nonlatt-meas},
and the constant
$\lambda(D)=\E\left(\sum_{i\in\T_1}|\ln r_i|\, r_i^D\right)$ from \eqref{limitconstant}.\\
\begin{thms}\label{maintheorem}
Let $k\in\{0,1,\ldots,d\}$ and $F$ be a self-similar random set with basic space $J$ satisfying the following conditions:
\begin{enumerate}[{\rm (i)}]
\item $1<\E \nu<\infty$,
\item the strong open set condition, i.e., the open set condition \eqref{OSC} and
      $$\P(F\cap \Int(J)\ne\emptyset)>0\, ,$$
\item $F$ is regular in the sense of Definition \ref{def:regular SSRS}, if $k\le d-2$,
\item for $k\le d-2$,
     $$\E\left(\esup_{0<r<\j}\limits\sup_{\s\in\T_b(r)}\, r^{-k}\,C_k^{var}\bigg(F(r),\partial (F_\s(r))\cap\partial\big(\bigcup_{\s'\in\T(r),~ \s'\ne\s}F_{\s'}(r)\big)\bigg)\right)<\infty\, .$$

\end{enumerate}
Set
$$R_k(r):=C_k(F(r))-\sum_{i\in\T_1}\1_{(0,\j r_i]}(r)\, C_k\big(F_i(r)\big)\, ,~~ r>0\, .$$
Then we have the following.
\begin{itemize}
\item[{\rm (I)}]
$$\elim_{\ep\rightarrow 0}\limits\ep^{D-k}\E\big(C_k(F(\ep))\big)=\frac{1}{\lambda(D)}\int_0^{\j} r^{D-k-1}\E\big(R_k(r)\big)\, dr$$
and
$$C_k(F):=\elim_{\ep\rightarrow 0}\limits\ep^{D-k} C_k(F(\ep))=\frac{M_\infty}{\lambda(D)}\int_0^{\j} r^{D-k-1}\E\big(R_k(r)\big)\, dr\quad  w.p.1\, ,$$
provided the measure $\mu$ is non-lattice.
\item[{\rm (II)}]
  $$\lim_{n\rightarrow\infty}e^{(k-D)(s+nc)}\E\left(C_k\big(F(e^{-(s+nc)})\big)\right)=\frac{1}{\lambda(D)}\sum_{m=0}^\infty e^{(k-D)(s+mc)}\E\left(R_k\big(e^{-(s+mc)}\big)\right)$$
   for a.e. $s\in[0,c)$ and
   $$\lim_{n\rightarrow\infty}e^{(k-D)(s+nc)}C_k\big(F(e^{-(s+nc)})\big)=\frac{M_\infty}{\lambda(D)}\sum_{m=0}^\infty e^{(k-D)(s+mc)}\E\left(R_k\big(e^{-(s+mc)}\big)\right)$$
 for a.e. $s\in[0,c)$ w.p.1,\\
 provided the measure $\mu$ has lattice constant $c$.\\
\end{itemize}
\end{thms}
\begin{remarks}
For $k=d$, i.e. for the (average) {\it Minkowski content} in (I), this theorem is due to Gatzouras \cite{Ga00}.\\
Recall that for $d\le 3$ the regularity (iii) holds always true. For polyconvex neighborhoods this remains valid for general $d$.\\
We conjecture that in the deterministic case the strong open set condition implies (iii). Under the additional assumption of polyconvex neighborhoods (iv) is proved implicitly in the Thesis of Winter \cite{Wi06}, \cite{Wi08}. For many deterministic examples with non-polyconvex neighborhoods, e.g. the Koch curve, the above conditions can be checked using their local structure.\\
For the general deterministic case and $k=d-1$ the limits are derived in Rataj and Winter \cite{RW09}. Moreover, these authors show for the case of non-arithmetic logarithmic contraction ratios the equality
$$ C_{d-1}(F)=(d-D)\, C_d(F)\, .$$
This supports our conjecture that like in the classical smooth case the parameters $C_k(F)$ with $k$ larger than the integer $[D]+1$ do not provide additional geometric information. (In the classical case they are all multiples of the Minkowski content.)
\end{remarks}
\vskip3mm

\begin{cors}\label{mainthm-average}
Under the conditions of Theorem \ref{maintheorem} we get the average limits
$$\lim_{\delta\rightarrow 0}\frac{1}{|\ln\delta|}\int_\delta^1\ep^{D-k}\E\big(C_k(F(\ep))\big)~\ep^{-1}d\ep = \frac{1}{\lambda(D)}\int_0^{\j} r^{D-k-1}\E\big(R_k(r)\big)\, dr$$
and
$$\lim_{\delta\rightarrow 0}\frac{1}{|\ln\delta|}\int_\delta^1\ep^{D-k} C_k(F(\ep))~ \ep^{-1} d\ep=\frac{M_\infty}{\lambda(D)}\int_0^{\j}r^{D-k-1}\E\big(R_k(r)\big)\, dr\quad  w.p.1\, .$$
\end{cors}
{\bf The notion of fractal curvatures.} In the non-lattice case these average limits agree with the ordinary limits from Theorem \ref{maintheorem} (I). Due to the stochastic self-similarity the randomness of the second limit appears only in form of the random variable $M_\infty$ which does not depend on $k$. In view of Biggin's theorem \ref{Biggins}, $M_\infty$ does not vanish with positive probability if and only if $$\mathbb{E}\bigg[\bigg(\sum_{i\in\T_1}r_i^D\bigg)\, \ln^+\bigg(\sum_{i\in\T_1}r_i^D\bigg)\bigg]<\infty$$
and in this case $\mathbb{E}M_\infty=1$. Then the expectation of the second limit agrees with the first limit, i.e. with the limit of the expectations. The integral expression provides a formula for numerical calculations in some special situations. Examples for the deterministic case may be found in Winter \cite{Wi08}.\\
In view of the classical notions the second limit in Corollary \ref{mainthm-average} will be called {\it random fractal Lipschitz-Killing curvature of order k} of the self-similar random set $F$, if it is not zero. \\
If the first limit vanishes, one has to check the correctness of the choice of the rescaling exponent $D-k$ in order to keep the curvature interpretation. In 'non-fractal' situations the exponent has to be replaced by $0$. For a detailed discussion of this problem see \cite{Wi08}. (Perhaps the rescaling exponents can be used in order to distinguish between 'fractal' and 'non-fractal' self-similar sets.)\\

\section{Proofs of the main results}
\subsection{Proof of Theorem \ref{maintheorem}}\label{ssec:proof-main-thm}
Recall that we wish to reduce the convergence assertions in (I) and (II) to the above renewal theorem.
Substituting $\ep:=\j e^{-t}$ and $r:=\j e^{-s}$ under the integral we obtain the equivalent limit relationships in Theorem \ref{renewal-branwalk} for the above introduced processes
\begin{eqnarray*}
Z_t & = & (\j)^{-k}e^{kt}\, C_k(F(\j e^{-t}))\1_{[0,\infty)}(t)\\
Y_t & = & Z_t - \sum_{i\in \T_1} Z_{t-W_i}
\end{eqnarray*}
with $W_i=|\ln r_i|$. (The use of the constant $\j$ will be clear later.) Therefore it suffices to check the conditions on the process $Y$.\\
The measurability and continuity properties in Condition \ref{a.e.continuousw.p.1} for the process $Y$ follow from its definition together with Lemmas \ref{lem:measurable-posreach} and \ref{lem:contcurvmeas}, and Corollary \ref{auxcontinuity}.\\
For (\ref{boundedness-condition}) it is sufficient to find some $\delta>0$ such that
$$\mathbb{E}\left[\esup_{t\ge 0}\limits\nolimits\limits\left(\frac{e^{-Dt}|Y_t|}{e^{-\delta t}}\right)\right]<\infty\, .$$
According to the above substitution this may be reformulated as
\begin{equation}\label{Q-boundedness}
\mathbb{E}\left[\esup_{0<r\le\j}\limits\left(\frac{|Q(r)|}{r^{k-D+\delta}}\right)\right]<\infty\, ,
\end{equation}
where
\begin{eqnarray*}
Q(r)&:=&C_k(F(r))-\sum_{i\in\T_1}r_i^k\, \1_{(0,\j]}\big(\frac{r}{r_i}\big)\, C_k\bigg(F^i\big(\frac{r}{r_i}\big)\bigg)\\
&=&C_k(F(r))-\sum_{i\in\T_1}\1_{(0,\j r_i]}(r)\, C_k\big(F_i(r)\big)\, =:\, R_k(r)
\end{eqnarray*}
The equality between $Q(r)$ (where the order $k$ of the curvature is suppressed in the notation) and $R_k(r)$ (which stands in the assertion of the theorem) follows from the scaling property of $C_k$.
(Recall that $F_i=F\cap J_i$, $J_i=S_iJ$, $F^i=S_i^{-1}(F)\cap J$, and $F(r)\subset\bigcup_{i\in\T_1}J_i(r)$.)\\
We decompose $Q(r)$ into
\begin{eqnarray*}
Q(r) & = & \left(C_k(F(r))-\sum_{i\in\T_1}r_i^k\, C_k\bigg(F^i\big(\frac{r}{r_i}\big)\bigg)\right)\\
&  & + \sum_{i\in\T_1}r_i^k\, \1_{(\j,\infty)}\big(\frac{r}{r_i}\big)\, C_k\bigg(F^i\big(\frac{r}{r_i}\big)\bigg)\\
& =: & Q_2(r)+Q_1(r)\, .
\end{eqnarray*}
Using that the $F^i$, $i\in\T_1$, are independent of $\T_1$ and have the same distribution as $F$ we obtain for $Q_1$ the estimate
$$\mathbb{E}\left[\esup_{0<r\le\j}\limits\left(\frac{|Q_1(r)|}{r^k}\right)\right]\le
\mathbb{E}\nu\, \mathbb{E}\left[\esup_{r>\j}\limits\left(\frac{|C_k(F(r))|}{r^k}\right)\right]\, .$$
According to Theorem \ref{thm-appendix} in the Appendix the last expression is finite. Therefore it remains to prove that
\begin{equation}\label{estimateQ_2}
\mathbb{E}\left[\esup _{0<r\le \j}\limits\left(\frac{|Q_2(r)|}{r^{k-D+\delta}}\right)\right]<\infty\,
\end{equation}
for some $\delta>0$. By the scaling property of $C_k$ we get
\begin{equation}\label{representationQ_2}
Q_2(r)=C_k(F(r))-\sum_{i\in\T_1}C_k\big(F_i(r)\big)\, .
\end{equation}

Next we decompose the total $k$th curvatures by means of the corresponding curvature measures:
\begin{eqnarray*}
C_k(F(r)) & = & C_k\bigg(F(r),\bigcup_{i\in\T_1}F_i(r)\bigg) = C_k\big(F(r), A_r\cup(A_r)^c\big)\\
& = & C_k(F(r),A_r)+C_k\big(F(r), (A_r)^c\big)\, ,
\end{eqnarray*}
where
\begin{equation}
A_r:=\bigcup_{j\ne k} J_j(r)\cap J_k(r)\, .
\end{equation}
Similarly,
$$C_k(F_i(r))= C_k(F_i(r),A_r)+C_k\big(F_i(r), (A_r)^c\big)\, ,\quad i\in\T_1\,. $$
The locality of the curvature measure $C_k$ implies
$$C_k(F_i(r),(A_r)^c)=C_k(F_i(r),B^i)=C_k(F(r),B^i)$$
and $F(r)\cap(A_r)^c$ is the disjoint union of the sets $B^i:=F_i(r)\setminus A_r$,\, $i\in\T_1$. Hence,
$$C_k(F(r),(A_r)^c))-\sum_{i\in\T_1}C_k(F_i(r),(A_r)^c)=0\, .$$
Substituting this in (\ref{representationQ_2}) we infer
$$Q_2(r)= C_k(F(r),A_r)-\sum_{i\in\T_1}C_k(F_i(r),A_r)$$
and by the scaling property of $C_k$ from this
\begin{eqnarray*}
Q_2(r) & =  & C_k(F(r),A_r)-\sum_{i\in\T_1}r^k_i\, C_k\bigg(F^i\big(\frac{r}{r_i}\big),S_i^{-1}(A_r)\bigg)\\
& = & Q_3(r)-Q_4(r)-Q_5(r)
\end{eqnarray*}
for
\begin{eqnarray*}
Q_3(r)& := & C_k(F(r),A_r)\\
Q_4(r)& := & \sum_{i\in\T_1}r_i^k\, \1_{(0,\j]}\big(\frac{r}{r_i}\big)\, C_k\bigg(F^i\big(\frac{r}{r_i}\big),S_i^{-1}(A_r)\bigg)\\
Q_5(r)& := & \sum_{i\in\T_1}r_i^k\, \1_{(\j,\infty)}\big(\frac{r}{r_i}\big)\, C_k\bigg(F^i\big(\frac{r}{r_i}\big),S_i^{-1}(A_r)\bigg)\, .
\end{eqnarray*}
Therefore it suffices to prove the estimate (\ref{estimateQ_2}) for $Q_3$, $Q_4$, and $Q_5$ instead of $Q_2$ separately.\\
First we obtain
$$\mathbb{E}\left[\esup _{0<r\le \j}\limits\left(\frac{|Q_5(r)|}{r^k}\right)\right]<\infty\, .$$
Here the arguments are the same as for $Q_1$ above taking into regard that for any Borel set $B$, $|C_k(F,B)|\le C_k^{var}(F,\mathbb{R}^d)$ and applying Theorem \ref{thm-appendix} in the Appendix.\\
For estimating $Q_3$ and $Q_4$ we we will use the set inclusions
$$A_r\subset (SJ)^c(r)\, ,\quad S_i^{-1}(A_r)\cap F^i\big(\frac{r}{r_i}\big)\subset J^c\big(\frac{r}{r_i})\, ,\quad {\rm and}~~ J^c(r)\subset(SJ)^c(r)\, ,$$
(Recall that $SJ=\bigcup_{i\in\T_1}S_iJ$ and $J$ is from OSC.) Then for $Q_3$ the estimate follows from (\ref{estimatemainlemma}) in Lemma \ref{mainlemma} below.\\
Finally, using once more that the sets $F^i$, $i\in\T_1$, are independent of $\T_1$ and have the same distribution as $F$, we infer for the same $\delta$ as above
\begin{eqnarray*}
\mathbb{E}\left[\esup _{0<r\le \j}\limits \left(\frac{|Q_4(r)|}{r^{k-D+\delta}}\right)\right]
& \le & \mathbb{E}\left[\sum_{i\in\T_1}r_i^{D-\delta}\esup_{0<\frac{r}{r_i}\le\j}\limits\left( \frac{r_i^{k-D+\delta}}{r^{k-D+\delta}}~ C_k^{var}\bigg(F^i\big(\frac{r}{r_i}\big), J^c\big(\frac{r}{r_i}\big)\bigg)\right)\right]\\
& \le &
\mathbb{E}\nu~ \mathbb{E}\left[\esup_{0<r\le \j}\limits\left(\frac{C_k^{var}\big(F(r), J^c(r)\big)}{r^{k-D+\delta}}\right)\right]\\
& \le &
\mathbb{E}\nu~ \mathbb{E}\left[\esup_{0<r\le \j}\limits\left(\frac{C_k^{var}\big(F(r), (SJ)^c(r)\big)}{r^{k-D+\delta}}\right)\right]
<\infty
\end{eqnarray*}
according to (\ref{estimatemainlemma}).\hfill $\Box$\\

\begin{lems}\label{mainlemma} Under the conditions of Theorem \ref{maintheorem}
we have
\begin{equation}\label{estimatemainlemma}
\mathbb{E}\left[\esup _{0<r\le \j}\limits\left(\frac{C_k^{var}\big(F(r),(SJ)^c(r)\big)}{r^{k-D+\delta}}\right)\right]<\infty
\end{equation}
for some $0<\delta<D$.
\end{lems}
\begin{proof}
We first proceed similarly as in the above  proof choosing the subtree  $\T(r)$ from \eqref{subtree-T(r)} instead of $\T_1$ in the decomposition of the curvature measures.
Recall that $F(r)=\bigcup_{\s\in\T(r)}F_\s(r)$
for any $r>0$. Therefore the locality of the curvature measures implies
\begin{equation}\label{auxestimate}
C_k^{var}\big(F(r), (SJ)^c(r)\big)=C_k^{var}\bigg(F(r), \big(\bigcup_{\s\in\T(r)}F_\s(r)\big)\cap(SJ)^c(r)\bigg)
\le\sum_{\s\in\T_b(r)}C_k^{var}(F(r), F_\s(r))\, ,
\end{equation}
where the subtree $\T_b(r)$ is defined in (\ref{subtree-T_b(r)}). For $k\in\{d-1,d\}$ we can use the (in)equality \begin{equation}\label{auxFs}
C_k^{var}\big(F(r),F_\s(r)\big)\le C_k\big(F_\s(r)\big)\, ,~~\s\in\T_b(r)\, .
\end{equation}
Furthermore, using that the curvature measures are concentrated on the boundary of $F(r)$ we obtain for any $\s\in\T_b(r)$ and $k\le d-2$,
\begin{eqnarray*}
& &C_k^{var}(F(r),F_\s(r))\\
&=&C_k^{var}\bigg(F(r),F_\s(r)\setminus\bigcup_{\s'\in\T(r),\, \s'\ne\s}F_{\s'}(r)\bigg)
+C_k^{var}\bigg(F(r),F_\s(r)\cap\bigcup_{\s'\in\T(r),\, \s'\ne\s}F_{\s'}(r)\bigg)\\
&=&C_k^{var}\bigg(F(r),F_\s(r)\setminus\bigcup_{\s'\in\T(r),\, \s'\ne\s}F_{\s'}(r)\bigg)+C_k^{var}\bigg(F(r),\partial F_\s(r)\cap\partial\big(\bigcup_{\s'\in\T(r),\, \s'\ne\s}F_{\s'}(r)\big)\bigg)\\
&\le& C_k^{var}(F_\s(r))+C_k^{var}\bigg(F(r),\partial F_\s(r)\cap\partial\big(\bigcup_{\s'\in\T(r),\, \s'\ne\s}F_{\s'}(r)\big)\bigg)\\
&=&\bar{r}^k_\s\,  C_k^{var}\bigg(F^\s\big(\frac{r}{\bar{r}_\s}\big)\bigg)+ C_k^{var}\bigg(F(r),\partial F_\s(r)\cap\partial\big(\bigcup_{\s'\in\T(r),\, \s'\ne\s}F_{\s'}(r)\big)\bigg).
\end{eqnarray*}
Recall that the sets $F^\s=S_\s^{-1}(F)\cap J$, $\s\in\T_b(r)$, are independent of the subtree $\T_b(r)$ and have the same distribution as $F$. Moreover, for $\s\in\T_b(r)$ we have $\frac{r}{\bar{r}_\s}\ge\j$. Therefore the above estimates yield
\begin{eqnarray*}
\mathbb{E}\left[\esup_{0<r\le\j}\limits\sup_{\s\in\T_b(r)}\left(r^{-k}\, C_k^{var}\big(F(r),F_\s(r)\big)\right)\right]
\le\mathbb{E}\left[\esup_{r\ge\j}\limits\left(r^{-k}\, C_k^{var}\big(F(r)\big)\right)\right]\\
+\mathbb{E}\left[\esup_{0<r\le\j}\limits\sup_{\s\in\T_b(r)}\left(r^{-k}\, C_k^{var}\bigg(F(r),\partial F_\s(r)\cap\partial\big(\bigcup_{\s'\in\T(r),\, \s'\ne\s}F_{\s'}(r)\big)\bigg)\right)\right]\, .\\
\end{eqnarray*}
The first summand on the right side is bounded by Theorem \ref{thm-appendix} in the Appendix, since $|J|\ge|F|$. (This holds also for $k\in\{d-1,d\}$.) The boundedness of the second summand is assumption (iv) in Theorem \ref{maintheorem}. In view of this, (\ref{auxestimate}) and (\ref{auxFs})
it suffices now to show that
\begin{equation}\label{estimatenumberT_b}
\mathbb{E}\bigg[\sup_{0<r\le\j}r^{D-\delta}\, \sharp(\T_b(r))\bigg]<\infty
\end{equation}
for some $0<\delta<D$, where $\sharp$ denotes the number of elements of a finite set.\\
For we use some ideas from the deterministic case (cf. \cite{Wi08}) and show a probabilistic version for our random subtrees $\T_b(r)$:\\
By the strong open set condition  $\mathbb{P}(F\cap\Int(J)\ne\emptyset)>0$ there exist some $\alpha>0$ and $0<\r<\j$ such that for the subtree
\begin{\eq}\label{defT(rhoalpha)}
\T(\r,\alpha):=\left\{\s\in\T(\r):\, d(x,\partial J)>\alpha\, , ~ x\in F_\s\right\}
\end{\eq}
we have
\begin{\eq}\label{Trhoalphanotempty}
\mathbb{P}(\T(\r,\alpha)\ne\emptyset)>0\, .
\end{\eq}
The Markov stopping property of the subtree $\T(\r)$ of $\T$ implies
$$\mathbb{E}\bigg[\sum_{\s\in\T(\r)}\bar{r}^D_\s\bigg] = 1$$
for the Hausdorff dimension D. Hence, there is a unique $0<\delta<D$ such that
\begin{\eq}\label{defdelta}
\mathbb{E}\bigg[\sum_{\s\in\T(\r)\setminus\T(\r,\alpha)}\bar{r}_\s^{D-\delta}\bigg] = 1\, .
\end{\eq}
For $r>0$ denote
$$\Xi(r):= \left\{\s=\s_1\ldots\s_n\in\T(r):\, \s_{k+1}\s_{k+2}\ldots\s_l\notin \T^{\s_1\ldots\s_k}(\r,\alpha)~\mbox{for any}~ 1\le k<l\le n\right\}\, .$$
(Recall the notation $X^\s(\o)=X(\theta_\s\o)$ for a random element $X$.)\\
As in the deterministic case (see \cite{Wi08}, proof of Lemma 5.4.1, Part I, where the polyconvex setting is not needed) one shows that
$$\sharp(\T_b(r))\le\sum_{i=1}^\nu \sharp\big(\Xi^i(r^*)\big)$$
where $r^*:=2(\alpha r_{min}\j)^{-1}r$. Consequently,
$$\mathbb{E}\bigg[\sup_{0<r<\j}\bigg(r^{d-\delta}\, \sharp(\T_b(r))\bigg)\bigg]\le\mathbb{E}\nu\,
\mathbb{E}\bigg[\sup_{r>0}\bigg(r^{d-\delta}\, \sharp(\Xi(r^*))\bigg)\bigg]\, $$
since the $\Xi^i$ have the same distribution as $\Xi$ and are independent of $\nu$.\\
Thus it is sufficient to prove that
\begin{\eq}\label{estimatenumberXi}
\mathbb{E}\bigg[\sup_{r>0}\bigg(r^{D-\delta}\, \sharp(\Xi(r))\bigg)\bigg]<\infty\, .
\end{\eq}
Because of the open set condition on the (deterministic) basic set $J$ for $r\ge\r'>0$ the number of its smaller copies under the random similarities of size of order $r$ and hence, the number of elements of $\T(r)$, is uniformly bounded. (Use a volume comparing argument for disjoint open balls, one in each copy, of the same radius.) Since $\sharp(\Xi(r))\le\sharp(\T(r))$ we obtain for any $0<\r'<\r$ a constant $C>0$ such that
$$\sup_{r\ge\r'}\bigg (r^{D-\delta}\sharp(\Xi(r))\bigg)\le C~~~{\rm w.p.1}$$
and from this
\begin{\eq}\label{numberXi}
\sup_{r>0}\bigg(r^{D-\delta}\sharp(\Xi(r))\bigg)\le\max\bigg(\sup_{0<r<\r'}\bigg(r^{D-\delta}\sharp(\Xi(r))\bigg), C\bigg)~~~{\rm w.p.1}\, .
\end{\eq}
On the other hand, for $0<r<\r$ the definition of the random set $\Xi(r)$ implies
$$\sharp(\Xi(r))\le \sum_{\s\in\T(\r)\setminus\T(\r,\alpha)}\sharp\bigg(\Xi^\s\big(\frac{r}{\bar{r}_\s}\big)\bigg)\, .
$$
Combining this with (\ref{numberXi}) for $\r':=\r$ we infer for any $0<r'\le\r$,
\begin{eqnarray*}
\Psi(r')&:= &\mathbb{E}\bigg[\max \bigg(\sup_{r'<r}\bigg(r^{D-\delta}\, \sharp(\Xi(r))\bigg), C\bigg)\bigg]\\
&= &\mathbb{E}\bigg[\max \bigg(\sup_{r'<r<\r}\bigg(r^{D-\delta}\, \sharp(\Xi(r))\bigg), C\bigg)\bigg]\\
&\le&\mathbb{E}\left[\max\left(\sup_{r'<r<\r}\left(\sum_{\s\in\T(\r)\setminus\T(\r,\alpha)}\bar{r}^{D-\delta}_\s
\frac{r^{D-\delta}}{\bar{r}^{D-\delta}_\s}\, \sharp\bigg(\Xi^\s\big(\frac{r}{\bar{r}_\s}\big)\bigg)\right), C\right)\right]\\
&\le&\mathbb{E}\left[\sum_{\s\in\T(\r)\setminus\T(\r,\alpha)}\bar{r}^{D-\delta}_\s
\max\bigg(\sup_{r'<r<\r}\bigg(\frac{r^{D-\delta}}{\bar{r}^{D-\delta}_\s}\, \sharp\bigg(\Xi^\s\big(\frac{r}{\bar{r}_\s}\big)\bigg)\bigg), C\bigg)\right]\\
&\le&\mathbb{E}\left[\sum_{\s\in\T(\r)\setminus\T(\r,\alpha)}\bar{r}^{D-\delta}_\s
\max\bigg(\sup_{r'<r}\bigg(\frac{r^{D-\delta}}{\bar{r}^{D-\delta}_\s}\, \sharp\bigg(\Xi^\s\big(\frac{r}{\bar{r}_\s}\big)\bigg)\bigg), C\bigg)\right]\\
&\le&\mathbb{E}\left[\sum_{\s\in\T(\r)\setminus\T(\r,\alpha)}\bar{r}^{D-\delta}_\s
\max\bigg(\sup_{r'\theta<r}\bigg(r^{D-\delta}\sharp(\Xi^\s(r))\bigg), C\bigg)\right]\\
& =& \mathbb{E}\left[\max\bigg(\sup_{r'\theta<r}\bigg(r^{D-\delta}\sharp(\Xi(r))\bigg), C\bigg)\right]= \Psi(r'\theta)
\end{eqnarray*}
with $\theta:=\r^{-1}\j>1$, since for $\s\in\T(\r)$ we have $\j\bar{r}_\s\le\r$, the random numbers $\sharp(\Xi^\s(r))\, ,~0<r<\r$, are independently of $\T(\r)\setminus\T(\r,\alpha)$  distributed as $\sharp(\Xi(r))$, and
$$\mathbb{E}\left[\sum_{\s\in\T(\r)\setminus\T(\r,\alpha)}\bar{r}^{D-\delta}_\s\right]=1$$
according to (\ref{defdelta}).
By monotonicity of the function $\Psi(r')$ we infer for any $0<r'<\r$, $\Psi(r')=\Psi(\theta r')$ and hence, $\lim_{r\rightarrow 0}\Psi(r)=\Psi(r')$,  which leads together with (\ref{numberXi}) to  assertion (\ref{estimatenumberXi}).
\end{proof}
\section{Appendix - An estimate for Lipschitz-Killing curvature measures of large parallel sets}
Here we will show that the variation of the $k$ th  Lipschitz-Killing curvature of the parallel set of amount r of an arbitrary compact set $K$ for sufficiently large $r$ is bounded from above by a constant multiple of $r^k$:\\

\begin{thm}\label{thm-appendix}
For any $R>\sqrt{2}$ and $k=0,1,\ldots,d$ there exists a constant $c_k(R)$ such that for any compact set $K\subset\mathbb{R}^d$ we have for any $r\ge R|K|$,
$$\rea(\widetilde{K(r)})\ge|K|\sqrt{R^2-1}\, ,$$
$\partial K(r)$ is a $(d-1)$-dimensional Lipschitz submanifold, and
$$\sup_{r\ge R|K|}\frac{C_k^{var}(K(r),\rd)}{r^k}\le c_k(R)\, .$$
\end{thm}
(It is well-known that for compact convex sets K these properties hold for all $R>0$, the last one in the sharper version of an isodiametric inequality with an optimal constant.)
\begin{proof}
According to an argument of Rataj (cf. Lemma 2.2 in Hug, Last and Weil \cite{HLW04}) we have for $r>|K|$, $$\rea\big(\widetilde{K(r)}\big)\ge\sqrt{r^2-|K|^2}\, ,$$ which implies the first assertion.\\
By the scaling property of the curvature measures we get for any $r>0$,
$$C_k^{var}(K(r),\rd)=r^k\, C_k^{var}((r^{-1}K)(1),\rd)\, .$$
Hence, it suffices to show for $R>\sqrt{2}$ the inequality
\begin{\eq}
\sup_{|K|\le R^{-1}} C_k^{var}(K(1),\rd)\le c_k(R)
\end{\eq}
and that for such $K$ the parallel set $K(1)$ has a Lipschitz boundary. For, we fix an arbitrary $0<s<\sqrt{1-R^{-2}}-R^{-1}$.  Since the curvature measures are translation invariant, we may assume that $0\in K$. Then we obtain for closed balls $B(0,r)$ with center $0$ and radius $r$ the set inclusions
$$B(0,1-s)\subset K(1)\subset B(0,1+|K|)\subset B(0,1+R^{-1})$$
and from this the following estimates for the Euclidean distance function $d(\cdot,\cdot)$.\\
$$s\le\inf_{x\in\partial B(0,1-s)}d\big(x,\widetilde{K(1)}\big)\, .$$
(If $d\big(x,\widetilde{K(1)}\big)<s$ for some $x\in\partial B(0,1-s)$, we get $d\big(0,\widetilde{K(1)}\big)\le d(0,x)+d\big(x,\widetilde{K(1)}\big)<1-s+s=1$ which  is a contradiction, since $0\in K$.) Furthermore, $d\big(x,\widetilde{K(1)}\big)\le d(x,\partial B(0,1+R^{-1}))\le 1+R^{-1}-(1-s)=R^{-1}+s$,
for any $x\in\partial B(0,1-s)$, which implies
$$\sup_{x\in\partial B(0,1-s)}d\big(x,\widetilde{K(1)}\big)\le R^{-1}+s<\sqrt{1-R^{-2}}\le\rea\big(\widetilde{K(1)}\big)\, .$$
Therefore we can define a biunique mapping from the unit normal bundle of the ball $B(0,1-s)$ onto that of the complementary set $\widetilde{K(1)}$ as follows:
$f_K: \nor B(0,1-s)\rightarrow\nor \widetilde{K(1)}$ with
$$f_K\bigg(x,\frac{x}{|x|}\bigg):=\bigg(\Pi_{\widetilde{K(1)}}x,\frac{x-\Pi_{\widetilde{K(1)}}x}{| x-\Pi_{\widetilde{K(1)}}x|}\bigg).$$
$f_K$ is a Lipschitz mapping whose Lipschitz constant is uniformly bounded in $K$ as above by some constant $c(R,s)$  depending only on  $R$ and $s$. This follows from Theorem 4.8 in \cite{Fe59}. (Federer's value $r$ is our $R^{-1}+s$, his $q$ is our $\sqrt{1-R^{-2}}$, and $s$ has the same meaning as there.) In particular, $K(1)$ has a Lipschitz boundary.\\
The transformation formula for rectifiable currents with compact support under Lipschitz mappings (cf. Federer \cite[4.1.30]{Fe69}) implies for the unit normal cycles
$$N_{\widetilde{K(1)}}(\varphi)=\int_{\nor\widetilde{K(1)}}\varphi=\int_{\nor B(0,1-s)}(f_K)^\sharp\varphi=N_{B(0,1-s)}\big((f_K)^\sharp\varphi\big)$$
for any smooth differential $(d-1)$-form $\varphi$ on $\rd\times\rd$, where $(f_K)^\sharp\varphi$ denotes its pullback form under the mapping $f_K$ (in the sense of a.e. differentiation).\\
Recall that for any set $X$ of positive reach and any bounded Borel set $B\subset\rd$ we have
$$C_k(X,B)=N_X\llcorner \1_{B\times\rd}(\varphi_k)$$
for the $k$th Lipschitz-Killing curvature form $\varphi_k$, $k=0,\ldots,d-1$. This leads to the estimates
\begin{\ea*}
C_k^{var}(K(1),\rd) &=&C_k^{var}\big(\widetilde{K(1)},\rd\big)\\
&\le&(\Lip(f_K))^{d-1}||N_{\nor B(0,1-s)}||\, ||\varphi_k\llcorner (B(0,1+R^{-1})\times\rd)||\\&\le& c(R,s)^{d-1}||N_{\nor B(0,1-s)}||\, ||\varphi_k\llcorner (B(0,1+R^{-1})\times\rd)||
\end{\ea*}
for the mass norm of rectifiable currents and the comass norm of differential forms restricted to $B(0,1+R^{-1})\times\rd$. (For notations and details on current theory see Federer \cite[Chapter 4]{Fe69}.) The minimum over $s$ as above on the right hand side (but also the expression for fixed $s$) provides a desired upper bound $c_k(R)$ if $k\le d-1$. The case $k=d$ is trivial.
\end{proof}

\bibliographystyle{amsplain}


\end{document}